\documentclass[11pt, a4paper]{scrartcl}

\usepackage[cmtip,arrow]{xy}
\usepackage{amsmath,amssymb,enumerate}
\usepackage{amsthm}

\def \C{{\mathbb C}}
\def \F{{\mathbb F}}

\def \N{{\mathbb N}}

\def \Q{{\mathbb Q}}
\def \R{{\mathbb R}}
\def \Re{\operatorname{Re}}
\def \tr{{ tr}}

\def \Z{{\mathbb Z}}
\def \l{\left}
\def \r{\right}
\def \f{\frac}
     
\newtheorem{thm}{Theorem}

\newtheorem{cor}{Corollary}

\title{Chebyshev's Bias\\
for Ramanujan's $\tau$-function \\via the Deep Riemann Hypothesis}
\author{Shin-ya Koyama and Nobushige Kurokawa}
\date{}

\begin{document}
\maketitle
\begin{abstract} 
The authors assume the Deep Riemann Hypothesis
to prove that a weighted sum of Ramanujan's $\tau$-function has a bias to being positive.
This phenomenon is an analogue of Chebyshev's bias.
\end{abstract}

\section{Deep Riemann Hypothesis}
Let $M(p)$ be a unitary matrix of degree $r\in\N=\{1,2,3,...\}$ defined for each prime number $p$.
Consider an $L$-function expressed by the Euler product
\begin{equation}\label{EP}
L(s,M)=\prod_{p:\text{ prime}}\det(1-M(p)p^{-s})^{-1}.
\end{equation}
The above is absolutely convergent in $\Re(s)>1$.
It is assumed in this paper that $L(s,M)$ has an analytic continuation
as an entire function on $\C$
and that a functional equation holds with $s\leftrightarrow1-s$.
Moreover, $L(\f12,M)\ne0$ is assumed for simplicity.

The following conjecture is an essential part
of the Deep Riemann Hypothesis 
named and proposed by the second author in \cite{K}:

\textbf{Deep Riemann Hypothesis (DRH):}
The Euler product \eqref{EP} converges at $s=\f12$ with
$$
\lim_{x \to \infty}\prod_{p \le x} \det\l(1-M(p)p^{-\f12}\r)^{-1} 
=\sqrt{2}^{\delta(M)} L\l(\tfrac12, M\r),
$$
where $\delta(M)=\mathrm{ord}_{s=1}L(s,M^2)$ 
with ord$_{s=1}$ signifying the order of the pole at $s=1$.

Here the facter $\sqrt 2$ in the right hand side was discovered by Goldfeld \cite{G},
who proved the DRH implies the GRH for the $L$-functions attached to elliptic curves $E$
and that the DRH implies $m=\mathrm{rank}(E)$.

\textit {Remark.}
Note that since
\begin{align*}
L(s,M^2)
&=\prod_{p:\text{ prime}}\det(1-M(p)^2p^{-s})^{-1}\\
&=\f{L(s,\mathrm{Sym}^2M)}{L(s,\wedge^2M)},
\end{align*}
it holds that
$$
\delta(M)=\mathrm{ord}_{s=1}L(s,\mathrm{Sym}^2M)-\mathrm{ord}_{s=1}L(s,\wedge^2M).
$$
Here we do not suppose that $M$ is a representation. 
The square $M^2$ is interpreted as the Adams operation.

\textbf{Example 1.}
When $r=1$ and $M$ is the non-principal Dirichlet character modulo 4, namely,
$M(p)=(-1)^{\f{p-1}2}=\chi_{-4}(p)$ for odd primes $p$, the following holds.
$$
L(s,M)=L(s,\chi_{-4})\text{ and } \delta(M)=1.
$$
\textbf{Example 2.}
Put $r=2$ and let $\tau(p)\in\Z$ be Ramanujan's $\tau$-function defined
for $q=e^{2\pi iz}$ with $z\in\C$ and $\mathrm{Im}(z)>0$ by
$$
\Delta(z)=q\prod_{k=1}^\infty (1-q^k)^{24}
=\sum_{n=1}^\infty \tau(n)q^n.
$$
Denote $M(p)=\begin{pmatrix}e^{i\theta(p)}&0\\0&e^{-i\theta(p)}\end{pmatrix}$,
where
$\theta(p)\in[0,\pi]\cong\mathrm{Conj}(SU(2))$ is defined as 
$\tau(p)=2p^{\f{11}2}\cos(\theta(p))$.
It holds that
$$
L(s,M)=\prod_p(1-2\cos(\theta(p))p^{-s}+p^{-2s})^{-1}
$$
and that $\delta(M)=-1$.

\textit {Remark.}
By conventional notation Ramanujan's $L$-function is defined as
$$
L(s,\Delta)=\sum_{n=1}^\infty \f{\tau(n)}{n^s},
$$
which does not satisfy our assumption, since its functional equation holds with $s\longleftrightarrow 12-s$.
However, a normalization $\lambda(n)=n^{-\f{11}2}\tau(n)$ leads to the $L$-function
$$
L\l(s+\f{11}2,\Delta\r)=\sum_{n=1}^\infty \f{\lambda(n)}{n^s}
$$
which satisfies our assumption.
Putting $\lambda(p)=2\cos(\theta(p))$ for any prime $p$ gives 
$L(s,M)=L\l(s+\f{11}2,\Delta\r)$ in Example 2.

Numerical evidence of DRH for various Dirichlet characters is provided in \cite{KKK0}.
It is known by Conrad \cite{C} that the DRH implies the Riemann Hypothesis (RH).
The logical relation between RH and DRH is interpretable
in terms of the error term in the prime number theorem.
Here, illustrate the case that $M=\chi$ is a Dirichlet character.

Let $E(x,\chi)$ be the error term in the prime number theorem.
Namely,
$$E(x,\chi)=\psi(x,\chi)-\begin{cases}0&(\chi\ne1)\\x&(\chi=1)\end{cases},$$
where the following is put
$$
\psi(x,\chi)=\sum\limits_{n\le x}\chi(n)\Lambda(n)
$$
with 
$$
\Lambda(n)=\begin{cases}\log p&(n=p^k)\\0&\text{(otherwise)}\end{cases}.
$$
The following table shows the estimates of $E(x,\chi)$ which are implied by
the region of convergence of the Euler product (EP) of $L(s,\chi)$.
The estimate of $E(x,\chi)$ is improved by extending the region.

\begin{table}[h!]\centering
\begin{tabular}{|c|c|}
\hline
If EP converges in & Then $E(x,\chi)$ is\\
\hline
$\Re(s)\ge 1$ (classical)& $o(x)$\\
$\Re(s)> \alpha $ & $O( x^\alpha (\log x)^2)$\\
$\Re(s)> 1/2$ (RH)& $O(\sqrt x(\log x)^2)$\\
$\Re(s)\ge 1/2$ (DRH)& $o(\sqrt x\log x)$\\
\hline
\end{tabular}
\end{table}

The first row presents the classical prime number theorem proved by Hadamard and de la Vall\'ee Poussin.
The second and the third rows represent basic theorems found in textbooks in analytic number theory.
The last row reflects the fact that the DRH is equivalent to the bound $o(\sqrt x\log x)$,
which is proved by Conrad \cite{C}.

To end this section, append a note on the DRH for the Riemann zeta function $\zeta(s)$,
which is the case of $r=1$ and $M(p)=1$.
It satisfies neither the assumption of entireness nor the original form of the DRH,
since the pole at $s=1$ prevents the Euler product from converging at any point in $\Re(s)<1$.
However, Akatsuka discovered an asymptotic behavior of the Euler product on the critical line
which exactly corresponds to the above bound $o(\sqrt x\log x)$.
Indeed he proved the following theorem.

\textbf {Akatsuka's Theorem \cite{A}:}\
Let 
$$\zeta_x(s)=\prod_{p\le x}(1-p^{-s})^{-1}$$
be the finite Euler product of the Riemann zeta function over
primes $p\le x$.
Then the following conditions (i)--(iii) are equivalent.
\begin{enumerate}[(i)]
\item Let $\psi(x):=\psi(x,1)$ with $ 1$ being the trivial character.
Then
$$\psi(x)=x+o(\sqrt x\log x)\qquad(x\to\infty).$$
\item There exists $\tau_0\in\R$ such that
$$
\lim_{x\to\infty}\frac{(\log x)^m\zeta_x(s_0)}
{\exp\left[\lim_{\varepsilon\downarrow0}\left(\int_{1+\varepsilon}^x\frac{du}{u^{s_0}\log u}-\log\frac1\varepsilon\right)\right]}
$$
converges to a nonzero limit,
where $m$ is the order of the zero for $\zeta(s)$ at $s=s_0=\frac12+i\tau_0$. 
\item The above limit exists and is nonzero for any $\tau_0\in\R$.
\end{enumerate}

Here his condition (ii) is called the \textit {Deep Riemann Hypothesis} for $\zeta(s)$.
Note that Akatsuka's work\footnote{
It is regretful that the review article in MathSciNet for Akatsuka's paper \cite{A} is misleading.
It erroneously says Akatsuka's theorem gives an equivalent condition to the RH.
The point is that the DRH is stronger than the RH.
}
refines and reformulates the pioneering results of Ramanujan \cite{R}.

\section{Chebyshev's Bias}
The use of DRH enables us to unveil the mystery of Chebyshev's bias.
This section describes how it is achieved according to Aoki-Koyama \cite{AK}.
Recalling their discussion will clarify the reason why our main theorem 
on Ramanujan's $\tau$-function is regarded as an analogue of Chebyshev's bias.

Chebyshev's bias is the phenomenon that
there seem to be tending to be more primes of the form $4k+3$ than of the form $4k+1$ $(k\in\Z)$.
In fact denoting by $\pi(x;\,q,\,a)$ the number of primes $p\le x$ such that $p\equiv a\pmod q$,
the inequality 
\begin{equation}\label{1}
\pi(x;\, 4,\, 3) \ge \pi(x;\, 4,\, 1)
\end{equation}
holds for any $x$ less than $26861$, 
which is the first prime number violating the inequality \eqref{1}.
However, the both sides draw equal at the next prime 26863, and $\pi(x;\, 4,\, 3)$ gets ahead again until
$616841$.
It is computed that more than ``97\% of $x$'' satisfy the inequality \eqref{1}
in  $\{x\in\R\ |\ 3\le x<10^{11}\}$ .

Littlewood \cite{L}, however, proved that the difference
$\pi(x;\, 4,\, 3) - \pi(x;\, 4,\, 1)$ changes its sign infinitely many times.
In 1962 Knapowski and Turan conjectured that the limit of the percentage in all positive numbers of the set
$$
A_X=\{x<X\ |\ \pi(x;\, 4,\, 3) \ge \pi(x;\, 4,\, 1)\}
$$
as $X\to\infty$ would equal 100\%,
but now it is proved in \cite{Kac} under the Generalized Riemann Hypothesis
that the limit does not exist and that the conjecture is false.

In place of such a naive density, the logarithmic density takes its place.
Define the \textit {logarithmic density} of the set $A_X$ in $[2,X]$ by
$$
\delta(A_X)=\f1{\log X}\int_{t\in A_X}\f{dt}t.
$$
Rubinstein and Sarnak \cite{RS} proved that the limit $\lim_{X\to\infty}\delta(A_X)$
exists and equals $0.9959...$
under the assumption of the Generalized Riemann Hypothesis and
the Grand Simplicity Hypothesis for $L(s,\chi)$,
which asserts linear independence over $\Q$ of the imaginary parts of all nontrivial zeros  of $L(s,\chi)$
in the upper half plane.

It is known by Dirichlet's prime number theorem in arithmetic progressions that
the number of primes of the form $4k+3$ and $4k+1$ should equal.
Therefore, Chebyshev's bias means that the primes of the form $4k+3$ appears ``earlier'' than those of the form $4k+1$.
One of the reasons why the logarithmic density is effective may be that
it treats contribution of smaller numbers as greater ones
with help of the factor $1/t$ in the integral.

Instead of considering the logarithmic density,
Aoki and Koyama \cite{AK} adopted a weighted counting function
$$
\pi_s(x;\,q,\,a)=
\sum_{\genfrac{}{}{0pt}{1}{p<x:\,\text{prime}}{p\equiv a\pmod q}}\f{1}{p^s}\qquad(s\ge0)
$$
generalizing $\pi(x;\,q,\,a)=\pi_0(x;\,q,\,a)$.
Here the smaller prime $p$ allows higher contribution to $\pi_s(x;\,q,\,a)$,
as long as we fix $s>0$.
The function $\pi_s(x;\,q,\,a)$ $(s>0)$ should be more appropriate 
than $\pi(x;\,q,\,a)$ to represent the phenomenon,
because it reflects the size of primes which $\pi(x;\,q,\,a)$ ignores.
Although the natural density of the set
$$
A(s)=\{x>0\ |\ \pi_s(x;\,4,\,3)-\pi_s(x;\,4,\,1)>0\}
$$
does not exist when $s=0$, they showed under the assumption of the DRH that
it would exist and equal to 1 when $s=\f12$, that is,
$$
\lim_{X\to\infty}\f1X\int_{t\in A(\f12)\cap[2,X]}dt=1.
$$
Indeed, they reached a more precise form in \cite{AK}:
\begin{equation}\label{CB}
\pi_{\f12}(x;\,4,\,3)-\pi_{\f12}(x;\,4,\,1)
\sim\f12\log\log x\quad
(x\to\infty),
\end{equation}
where $f(x)\sim g(x)$ $(x\to\infty)$ means
$\lim\limits_{x\to\infty}\f{f(x)}{g(x)}=1$.

Once assuming DRH for $L(s,\chi_{-4})$, the proof of \eqref{CB} is 
simply illustrated as follows.
By DRH the limit
$$
\lim_{x \to \infty} \prod_{p \le x} \l(1-\chi(p)p^{-\f12}\r)^{-1}
$$
exists and is nonzero. Then it has a bounded logarithm:
$$
\sum_{p \le x} \log\l(1-\chi(p)p^{-\f12}\r)^{-1}=O(1)
\quad(x\to\infty).
$$
When we expand the left hand side as
$$
\sum_{k=1}^\infty\sum_{p \le x}\f{\chi(p)^k}{k p^{\frac k2}},
$$
the subseries over $k\ge3$ is absolutely convergent as $x\to\infty$, because we estimate that
\begin{align*}
\sum_{p\le x}\sum_{k=3}^\infty\l|\f{\chi(p)^k}{kp^{\frac k2}}\r|
&\le\f13\sum_{p\le x}\sum_{k=3}^\infty\f1{p^{\frac k2}}
\le\zeta\l(\f32\r)
\end{align*}
with the Riemann zeta function $\zeta(s)$ by an easy calculation
$$
\sum_{k=3}^\infty\f1{p^{\frac k2}}
=\f{\sqrt p}{\sqrt p-1}\f1{p^{\frac 32}}
\le\f3{p^{\frac 32}}
$$
for $p\ge3$.
On the other hand, the subseries over $k=2$ satisfies
by Euler's theorem that
$$
\sum_{p\le x}\f{\chi(p)^2}{2p}
=\sum_{p\le x}\f{1}{2p}
=\f12\log \log x +O(1)\quad (x\to\infty).
$$
Hence the behavior of the remaining part $k=1$ is obtained:
$$
\sum_{p\le x}\f{\chi(p)}{\sqrt p}=
-\f12\log \log x +O(1)\quad (x\to\infty).
$$
This completes the proof of \eqref{CB}.

In their original paper \cite{AK} Aoki and Koyama find that the bias described by \eqref{CB} is a special case of
the bias towards non-splitting primes in abelian extensions,
under the assumption of the DRH for $L$-functions attached to algebraic Hecke characters.

\section{Results}
The following theorem is a suitable generalization of \eqref{CB}
for the purpose of treating $\tau(p)$, which is the goal of this paper.

\begin{thm}\label{main}
Suppose DRH for $L(s,M)$, then
$$
\lim_{x\to\infty}\f{\displaystyle\sum\limits_{p\le x}\f{\tr(M(p))}{\sqrt p}}{\log\log x}
=-\f{\delta(M)}2.
$$
\end{thm}

\begin{proof}
\begin{align*}
\text{I}(x)   & = \sum_{p\le x} \frac{\tr(M(p))}{\sqrt p}, \\
\text{II}(x)  & = \f12 \sum_{p\le x} \f{\tr(M(p)^{2})}p, \\
\intertext{and}
\text{III}(x) & = \sum_{k \ge 3} \frac{1}{k} \sum_{p\le x}\f{\tr(M(p)^{k})}{p^{k/2}}. 
\end{align*}
Then,
$$
\text{I}(x)+\text{II}(x)+\text{III}(x)
=\log \prod_{p \le x} \det\l(1-M(p)p^{-\f12}\r)^{-1}.
$$
Hence DRH implies
\begin{equation}\label{K1}
\text{I}(x)+\text{II}(x)+\text{III}(x)
=O(1)\quad(x\to\infty).
\end{equation}
The generalized Mertens theorem (\cite{KKK}) says that
\begin{equation}\label{K2}
\lim_{x\to\infty}\f{\text{II}(x)}{\log\log x}
=\f{\delta(M)}2.
\end{equation}
Moreover, it is easy to see that
\begin{equation}\label{K3}
\text{III}(x)
=O(1)\quad(x\to\infty).
\end{equation}
Thus \eqref{K1} \eqref{K2} and \eqref{K3} give
$$
\lim_{x\to\infty}\f{\text{I}(x)}{\log\log x}
=-\f{\delta(M)}2.
$$
\end{proof}

Applying Theorem \ref{main} to Example 1 restores (3).
On the other hand, Example 2 together with Theorem 1 leads us to our main theorem:

\begin{thm}
Assume the DRH for $L(s+\f{11}2,\Delta)$. The following holds.
$$
\sum_{\genfrac{}{}{0pt}{1}{p\le x}{\text{prime}}}\f{\tau(p)}{p^6}\sim\f12\log\log x
\quad(x\to\infty).
$$
\end{thm}

\begin{proof}
By using the notation in Example 2, we deduce from Theorem 1 that
\begin{align*}
\sum_{\genfrac{}{}{0pt}{1}{p\le x}{\text{prime}}}\f{\tau(p)}{p^6}
&=\sum_{\genfrac{}{}{0pt}{1}{p\le x}{\text{prime}}}\f{2\cos(\theta(p))}{\sqrt p}\\
&=\sum_{\genfrac{}{}{0pt}{1}{p\le x}{\text{prime}}}\f{\tr(M(p))}{\sqrt p}\\
&\sim\f12\log\log x\quad(x\to\infty).
\end{align*}
\end{proof}

Theorem 2 asserts that the weighted sum of Ramanujan's $\tau$-function tends to be positive
under the DRH.
Sarnak \cite{S} also reached a similar prediction under the assumption of
the Generalized Riemann Hypothesis as well as the Grand Simplicity Hypothesis for $L(s,\Delta)$,
which asserts linear independence over $\Q$ of the imaginary parts of all nontrivial zeros of $L(s,\Delta)$
in the upper half plane.
He has pointed out that the sum
$$
S(x)=\sum_{\genfrac{}{}{0pt}{1}{p\le x}{\text{prime}}}\f{\tau(p)}{p^\f{11}2}
$$
has a bias to being positive, in the sense that the mean of the measure $\mu$ defined by
$$
\f1{\log X}\int_2^Xf\l(\f{\log x}{\sqrt x}S(x)\r)\f{dx}x\to\int_\R f(x)d_\mu(x)\quad
(x\to\infty)
$$
for $f\in C(\R)$ is equal to 1.
In the proof he closely examines the logarithmic derivative of $L(s,\Delta)$
to find that the second term in its expansion is the cause of the bias.
While our above discussion deals with the logarithm instead of its derivative,
we have also reached the point that the bias derives from the second term in the expansion.
Although our conclusion has a common cause of the bias with what is given in \cite{S},
our proof is straightforward enough to simplify the proofs.
Actually we have the following concise result.

\begin{cor}
Assume the DRH for $L(s+\f{11}2,\Delta)$.
The natural density of the set
$$
A=\l\{x>0\ \l|\ \sum_{p\le x}\f{\tau(p)}{p^6}>0\r\}\r.
$$
exists and is equal to 1. More precisely,
$$
\lim_{X\to\infty}\f1X\int_{t\in A\cap[2,X]}dt=1.
$$
\end{cor}

Sarnak \cite{S} also presented a prediction on the tendency of the
signature of $a(p)=p+1-\# E(\F_p)$ 
for an elliptic curve $E$ over $\Q$.
Under the Generalized Riemann Hypothesis and the Grand Simplicity Hypothesis,
he has observed that a weighted sum of $a(p)$ would have a bias to being
negative for $\mathrm{rank}(E)>0$ and positive for $\mathrm{rank}(E)=0$.

Our method introduced in this paper would provide evidence for this conjecture as well.
Our work on a bias for Satake parameters for $GL(n)$ in a more general framework
will include it in addition to a generalization of Theorem 2.
This will be done in a forthcoming paper \cite{KKK2}.

\end{document}